\documentclass{amsart}
\usepackage{amssymb}
\usepackage{enumitem}

\newtheorem{thm}{Theorem}[section]
\newtheorem{cor}[thm]{Corollary}
\newtheorem{lem}[thm]{Lemma}
\newtheorem{prop}[thm]{Proposition}

\numberwithin{equation}{section}

\DeclareMathOperator\Aut{Aut}
\DeclareMathOperator\ch{char}
\DeclareMathOperator\Span{Span}

\newcommand{\A}{\mathcal{A}}
\newcommand{\cnt}{\mathcal{Z}}
\newcommand{\gk}{\text{GK.dim}}
\newcommand{\inv}{^{-1}}
\newcommand{\iso}{\cong}
\newcommand{\K}{K}
\newcommand{\M}{\mathcal{M}}
\newcommand{\sym}{\mathcal{S}}

\newcommand{\qp}{\mathcal{O}_q(\K^2)}
\newcommand{\qpp}{\mathcal{O}_p(\K^2)}

\newcommand{\wa}{A_1^q(\K)}
\newcommand{\wap}{A_1^p(\K)}
\newcommand{\bp}{\mathbf{p}}
\newcommand{\bq}{\mathbf{q}}
\newcommand{\qnp}{\mathcal{O}_{\bp}(\K^n)}
\newcommand{\qnq}{\mathcal{O}_{\bq}(\K^n)}
\newcommand{\qmq}{\mathcal{O}_{\bq}(\K^m)}
\newcommand{\qmn}{\mathcal{O}_q(\M_n(\K))}
\newcommand{\pmn}{\mathcal{O}_p(\M_n(\K))}

\newcommand{\pma}{\mathcal{O}_{\lambda,\bp}(\M_n(\K))}
\newcommand{\jmn}{\mathcal{O}_J(\M_2(\K))}

\begin{document}

\title{Isomorphisms of some quantum spaces}

\author{Jason Gaddis}
\address{Department of Mathematical Sciences, University of Wisconsin - Milwaukee, Milwaukee, WI 53201}
\email{jdgaddis@uwm.edu}


\date{}

\begin{abstract}We consider a series of questions that grew out of determining when two quantum planes are isomorphic. In particular, we consider a similar question for quantum matrix algebras and certain ambiskew polynomial rings. Additionally, we modify a result by Alev and Dumas to show that two quantum Weyl algebras are isomorphic if and only if their parameters are equal or inverses of each other.\end{abstract}

\maketitle

\section{Introduction}

Quantum rigidity says that automorphism groups of quantum spaces should be small in some sense. Analogously, there should be relatively few isomorphisms between quantum spaces of the same type. In this paper we study the isomorphism problem for quantum matrix algebras, certain ambiskew polynomial rings, and quantum Weyl algebras.

It can be shown that two quantum planes, $\qpp$ and $\qp$, are isomorphic if and only if $p=q^{\pm 1}$. There are multiple approaches to this proof. If one considers only graded isomorphisms, then the result follows by considering $\qpp$ and $\qp$ as geometric algebras (see \cite{mori}). In the case that $p$ and $q$ are not roots of unity, Alev and Dumas proved this result by considering an invariant of the quotient division ring (\cite{adcorps}, Corollary 3.11). Our results rely on the linear algebra of graded algebras. While more computational, this allows one to handle the root of unity and nonroot of unity case simultaneously.

Throughout, $\K$ is a field and all algebras are $\K$-algebras. Isomorphisms should be read as `isomorphisms as $\K$-algebras'. An algebra is said to be \textit{graded} (or $\mathbb{N}$-graded) if $A$ has a direct sum decomposition $A = \bigoplus_{d \in \mathbb{N}} A_d$ by abelian groups and $A_d A_e \subset A_{d+e}$. An element $a \in A_d$ is said to be \textit{homogeneous} with degree $d$. If $A_0 = \K$, then $A$ is said to be \textit{connected graded}. If $A_1$ generates $A$ as an algebra, then $A$ is said to be \textit{generated in degree $1$} and a basis for $A_1$ is a \textit{generating basis} for $A$. If $A_1$ is finite-dimensional, then $A$ is said to be \textit{affine}. All algebras considered in this paper are affine connected graded and generated in degree $1$ with the exception of the quantum Weyl algebras.

If $R$ is an affine connected graded algebra and $a \in R$, then we can decompose $a$ into its homogeneous components, $a=a_0+\cdots+a_n$, $a_d \in A_d$. If $\Phi:R \rightarrow S$ is a map between affine connected graded algebras and $x_i$ a generating element of $R$, we denote by $\Phi_d(x_i)$ the homogeneous degree $d$ component of the image of $x_i$ under $\Phi$. We frequently make use of the graded structure and defining relations of the various algebras. By $T(i,j)$ we mean the image of the defining relation determined by $x_i$ and $x_j$ under $\Phi$ written as an expression in terms of the various $\Phi(x_k)$. Note that, if $\Phi$ is an isomorphism, then $T(i,j)=0$. In particular, because $T(i,j)$ lies in $S$, then each graded component $T_d(i,j)$ is zero. We will exploit this fact throughout.

The definitions presented below are well-known and there are many excellent references. Our primary source is \cite{browngood}.

\subsection*{Quantum matrix algebras}Let $p \in \K^\times$. The single parameter quantum matrix algebra $\pmn$ has generating basis $\{X_{ij}\}$, $1 \leq i,j \leq n$, subject to the relations
\begin{align*}
	X_{ij}X_{lm} = 
		\begin{cases}
			pX_{lm}X_{ij} & i > l, j = m \\
			pX_{lm}X_{ij} & i = l, j > m \\
			X_{lm}X_{ij} & i > l, j < m \\
			X_{lm}X_{ij} + (p-p\inv)X_{im}X_{lj} & i > l, j > m.
	 	\end{cases}
\end{align*}
Many authors use different relations which amount to swapping $p, p\inv$. The isomorphism result here is identical to that for the quantum planes (Proposition \ref{spma}). 

We say $\bq = (q_{ij}) \in\M_n(\K^\times)$ is \textit{multiplicatively antisymmetric} if $q_{ii} = 1$ and $q_{ij}=q_{ji}\inv$ for all $i \neq j$. Let $\mathcal{A}_n \subset \M_n(\K^\times)$ be the subset of multiplicatively antisymmetric matrices. The multi-parameter quantum $n\times n$ matrix algebra, $\pma$, has generating basis $\{X_{ij}\}$, $1 \leq i,j \leq n$, with parameters $\lambda \in \K^\times$ and $\bp \in \mathcal{A}_n$ subject to the relations
\begin{align*}
	X_{ij}X_{lm} = 
		\begin{cases}
			p_{il}p_{mj}X_{lm}X_{ij} + (\lambda-1)p_{il}X_{lj}X_{im} & i > l, j > m\\
			\lambda p_{il}p_{mj} X_{lm}X_{ij} & i > l, j \leq m \\
			p_{mj}X_{lm}X_{ij} & i = l, j > m.			
	 	\end{cases}
\end{align*}
Because of the parameter $\lambda$, we do not expect a result as simple as that for the single parameter case. However, we can provide a related result for the case of $n=2$.

\subsection*{Certain ambiskew polynomial rings}In \cite{jordan}, Jordan defines a class of iterated skew polynomial rings with generating basis $\{x_1,x_2,x_3,x_4\}$ and parameters $a,b,p_1,p_2 \in \K^\times$ subject to the relations
\begin{align*}
	x_4x_1 &= ax_1x_4 & & x_2x_1 = p_1\inv a\inv x_1x_2 & & x_1x_3 = p_1 x_3x_1 \\
	x_4x_3 &= bx_3x_4 & & x_2x_3 = p_1b\inv x_3x_2 & &x_2x_4 = p_2x_4x_2 + (1-p_2ab)x_1x_3. 
\end{align*}
Denote these algebras by $R(a,b,p_1,p_2)$. Making the identifications $x_1=\lambda q\inv X_{12}$, $x_2 = X_{22}$, $x_3 = X_{21}$, $x_4 = X_{11}$, we see that $R(q\inv,\lambda\inv q,\lambda\inv q^2,1)$ is isomorphic to $\mathcal{O}_{\lambda,\bq}(\M_2(\K))$ where $q_{12}=q$. In Section \ref{ambi}, we give necessary and sufficient conditions for two rings of the form $R(a,b,p_1,1)$ to be isomorphic under certain hypotheses.

\subsection*{Jordan matrix algebra}There is an additional `quantum matrix algebra' corresponding to the Jordan plane. As defined in \cite{manin2}, the algebra $\jmn$ has generating basis $\{x_1,x_2,x_3,x_4\}$ subject to the relations
\begin{align*}
	0 	&= [x_1,x_3] + x_3^2
		=  [x_1,x_2] - x_1x_3 + x_1x_4 - x_2x_3 - x_1^2 \\
		&= [x_4,x_3] + x_3^2
		=  [x_2,x_4] + x_1x_3 - x_1x_4 + x_2x_3 + x_4^2 \\
		&=  [x_2,x_3] + x_1x_3 + x_3x_4
		=  [x_1,x_4] - x_1x_3 + x_3x_4 - x_3^2.
\end{align*}
We show that this algebra is not isomorphic to the ambiskew polynomial rings above (Proposition \ref{jmn}) and therefore not isomorphic to the previously defined quantum matrix algebras.

\subsection*{Quantum Weyl algebras}The quantum Weyl algebra, $\wa$, is generated by two elements $x$ and $y$, subject to the relation $xy-qyx=1$, $q \in \K^\times$. It is affine and generated in degree 1 but not graded. Instead, the algebra has a filtration by subspaces $W_d = \{y^i x^j \mid i,j \in \mathbb{N}, i+j \leq d\}$. Then $W_d \subset W_{d+1}$, $W_d W_e \subset W_{d+e}$, and $\bigcup_d W_d = \wa$. 

We prove that $\wap \iso \wa$ if and only if $p=q^{\pm 1}$. Our proof of this theorem is split into two propositions (Proposition \ref{wiso1} and \ref{wiso2}). This result was proved recently in greater generality in \cite{suarez} in the context of quantum generalized Weyl algebras. We offer a different approach, by adapting the proof of Proposition 1.5 in \cite{alev1} by Alev and Dumas.

\vspace{1em}

In the appendix (Section \ref{qas}), we utilize the results of this paper to prove an isomorphism result for quantum affine spaces. For additional applications, see \cite{gad-thesis}.

\section{General results}

Throughout this section, let $\Phi:R \rightarrow S$ be a (not necessarily graded) isomorphism between affine connected graded algebras. Let $\{x_i\}$ (resp. $\{y_i\}$) be a generating basis for $R$ (resp. $S$) and suppose $1 \leq i \leq n$ in both cases. Our general strategy is to consider the image of certain defining relations under $\Phi$. The images of the generators can be controlled to a great degree by the graded structure on these algebras.

\begin{lem}\label{deg1}The degree 1 components of $\Phi(x_1),\hdots,\Phi(x_n)$ are all $\K$-linearly independent. Moreover, $\Phi_1$ maps $R_1$ isomorphically onto $S_1$.\end{lem}

\begin{proof}The isomorphism $\Phi$ is completely determined by its action on the $x_i$. Hence, the elements $\{\Phi(x_i)\}$ generate all of $S$. Let $f_i \in R$ such that $y_i = \Phi(f_i)$, $i \in \{1,\hdots,n\}$. Since $\deg(y_i)=1$, then $y_i = \Phi_1(f_i)$. 

Because $S$ is graded, then $\Phi_2(x_i)\cdot\Phi_d(x_j) \in S_{d+2}$. Moreover, since $S$ is connected graded, then $\Phi_0(x_i) \in S_0 = \K$. Let $r=x_{i_1}\cdots x_{i_m}$ be an arbitrary monomial in $R$. Then
\begin{align*}
	\Phi_1(r) = \left( \prod_{k=1}^m \Phi(x_{i_k}) \right)_1 = \left( \prod_{k=1}^m \Phi_0(x_{i_k}) + \Phi_1(x_{i_k}) \right)_1.
\end{align*}
Thus, we can write,
	\[ y_i = \sum_{j=1}^n \alpha_{ij} \Phi_1(x_j), ~\alpha_{ij} \in \K.\]
Hence $\Phi_1:R_1 \rightarrow S_1$ is onto. Moreover, $\dim_{\K}(R_1)=\dim_{\K}(S_1)$ and so $\Phi_1$ is an isomorphism.\end{proof}

The next step is to show that the constant term of the image of each generator is zero. This need not always hold, but it does in the generic case.

\begin{lem}\label{dzero}If $i,j \in \{1,\hdots,n\}$ such that $x_ix_j-px_jx_i=0$ for some $p \in \K^\times$, $p\neq 1$, then $\Phi_0(x_i)=\Phi_0(x_j)=0$.\end{lem}

\begin{proof}Without loss of generality, suppose $\Phi_0(x_i) \neq 0$. Let $T=\Phi(x_i)\Phi(x_j)- p \Phi(x_j)\Phi(x_i)$. Then $T_0=\Phi_0(x_i) \Phi_0(x_j) (1-p)=0$, so $\Phi_0(x_j) = 0$. Thus, $T_1=\Phi_0(x_i)\Phi_1(x_j)(1-p)=0$. Since $\Phi_1(x_j) \neq 0$ by Lemma \ref{deg1}, then $T_1 \neq 0$, a contradiction.\end{proof}

\section{Quantum matrix algebras}

By \cite{browngood}, Lemma II.9.7, $\gk(\pmn)=n^2$. Hence, $\pmn \iso \mathcal{O}_q(\M_m(\K))$ implies $m=n$. Let $\{X_{ij}\}$ (resp. $\{Y_{ij}\}$) be a generating basis for $\pmn$ (resp. $\qmn$). Throughout, we assume $n \geq 2$.

\begin{prop}\label{spma}The single parameter quantum matrix algebras $\pmn$ and $\qmn$ are isomorphic if and only if $p=q^{\pm 1}$.\end{prop}

\begin{proof}That $\pmn \iso \qmn$ when $p=q^{\pm 1}$ follows from \cite{parwang}, Remark 3.7.2. We prove the converse here.

Since $\mathcal{O}_1(\M_n(\K))$ is commutative, then $\mathcal{O}_1(\M_n(\K)) \iso \qmn$ implies $q=1$. Suppose $p,q\neq 1$. Write $\Phi_1(X_{22})= \sum a_{rs} Y_{rs}$ and $\Phi_1(X_{12})=\sum b_{rs} Y_{rs}$. By Lemma \ref{dzero}, $\Phi_0(X_{22})=\Phi_0(X_{12})=0$. Let $T= T((2,2),(1,2))$. Then,
\begin{align*}
	T_2 &= \Phi_1(X_{22})\Phi_1(X_{12}) - p\Phi_1(X_{12})\Phi_1(X_{22}) \\
		&= (1-p)\left( \sum_{1 \leq i,j \leq n} a_{ij}b_{ij} Y_{ij}^2 \right) + (1-p) \sum_{i>l, j>m} (a_{ij}b_{lm} + a_{lm}b_{ij}) Y_{lm}Y_{ij} \\	
		&~~~~~+ \sum_{\begin{subarray}{l} i > l, j=m \\ i=l, j > m \end{subarray}} \left( (q-p)a_{ij}b_{lm} + (1-pq) a_{lm}b_{ij} \right) Y_{lm} Y_{ij}\\
		&~~~~~+ \sum_{i > l, j < m} \left( (1-p)(a_{ij}b_{lm} + a_{lm}b_{ij}) + (q-q\inv)(a_{im}b_{lj}-pa_{lj}b_{im})\right) Y_{lm}Y_{ij}.
\end{align*}
The coefficients of the $Y_{ij}^2$ being zero imply that, for all $(i,j)$, either $a_{ij}= 0$ or $b_{ij}=0$. If $i>l$ and $j>m$, then the coefficient of $Y_{lm}Y_{ij}$ is $(1-p)(a_{ij}b_{lm} + a_{lm}b_{ij})=0$. One of $a_{ij}b_{lm}$, $a_{lm}b_{ij}$ must be zero, which implies that either they are both zero or $p=1$. The latter case contradicts our hypothesis. Thus, $a_{ij}b_{lm}=a_{lm}b_{ij}=0$ for all $i>l$, $j>m$. It then follows that if $i > l$ and $j < m$, then $a_{lj}b_{im}-pa_{im}b_{lj}=0$. Hence,
\begin{align*}
	T_2 &= \sum_{\begin{subarray}{l} i > l, j=m \\ i=l, j > m \end{subarray}} \left( (q-p)a_{ij}b_{lm} + (1-pq) a_{lm}b_{ij} \right) Y_{lm} Y_{ij}\\
		&~~~~~+ (1-p) \sum_{i > l, j < m} (a_{ij}b_{lm} + a_{lm}b_{ij}) Y_{lm}Y_{ij}	.
\end{align*}
Similar logic to the above shows that $a_{ij}b_{lm}=a_{lm}b_{ij}=0$ when $i > l$ and $j < m$. Therefore,
\begin{align*}
	T_2 &= \sum_{\begin{subarray}{l} i > l, j=m \\ i=l, j > m \end{subarray}} \left( (q-p)a_{ij}b_{lm} + (1-pq) a_{lm}b_{ij} \right) Y_{lm} Y_{ij}.
\end{align*}
By Lemma \ref{deg1}, there exists $(i,j) \neq (l,m)$ such that $a_{ij},b_{lm} \neq 0$. It now follows easily that either $p=q$ or $p=q\inv$.\end{proof}

\section{Certain ambiskew polynomial rings}\label{ambi}

We now consider the ambiskew polynomial rings defined in the introduction. Throughout this section, let $\{x_i\}$ (resp. $\{y_i\}$) be a generating basis for $R(a,b,p_1,p_2)$ (resp. $R(c,d,q_1,q_2)$).

\begin{prop}\label{amb-maps}Suppose $(a,b,p_1,p_2)$ is one of the following tuples:
\begin{enumerate}[label={(\arabic*)}]
	\item $(c,d,q_1,q_2)$,
	\item $(q_1\inv c\inv, q_1d\inv, q_1, q_2\inv)$,
	\item $(d,c,q_1\inv, q_2)$,
	\item $(q_1 d\inv, q_1\inv c\inv, q_1\inv, q_2\inv)$.
\end{enumerate}
Then $R(a,b,p_1,p_2) \iso R(c,d,q_1,q_2)$.
\end{prop}

\begin{proof}We define a rule $\Phi:R(a,b,p_1,p_2) \rightarrow R(c,d,q_1,q_2)$ in each case by
\begin{enumerate}[label={(\arabic*)}]
	\item $x_1 \mapsto 		y_1, 	x_2 \mapsto y_2, x_3 \mapsto y_3, x_4 \mapsto y_4$,
	\item $x_1 \mapsto cd    y_1, 	x_2 \mapsto y_4, x_3 \mapsto y_3, x_4 \mapsto y_2$,
	\item $x_1 \mapsto q_1   y_3, 	x_2 \mapsto y_2, x_3 \mapsto y_1, x_4 \mapsto y_4$,
	\item $x_1 \mapsto q_1cd y_3, 	x_2 \mapsto y_4, x_3 \mapsto y_1, x_4 \mapsto y_2$.
\end{enumerate}
We leave it to the reader to verify that these images indeed satisfy the defining relations of $R(a,b,p_1,p_2)$ and therefore extend to bijective homomorphisms.\end{proof}
 
At the present time, we are most interested in the multi-parameter quantum matrix algebras. Hence, we take $p_2,q_2=1$. Then there is no confusion in writing $p=p_1$ and $q=q_1$. Moreover, we assume that $a,b,ab,p^2,pa,pb\inv,pa^2,p\inv b^2 \neq 1$ (and similarly for the $c,d,q$). These last two requirements, in terms of the matrix algebras, both translate to $\lambda \neq 1$.

\begin{prop}With the above hypotheses, if $\Phi:R(a,b,p,1) \rightarrow R(c,d,q,1) $ is an isomorphism, then $(a,b,p,1)$ is one of (1)-(4) in Proposition \ref{amb-maps}.\end{prop}

\begin{proof}By Lemma \ref{dzero} and our hypotheses on the parameters, $\Phi_0(x_i)=0$ for each $i$, $i\in\{1,\hdots,4\}$. Write, $\Phi_1(x_i) = \sum_{j=1}^4 \alpha_{ij} y_j$. Then
\begin{align*}
	T_2&(1,3) = (1-p) \sum_{d=1}^4 \alpha_{1d}\alpha_{3d} y_d^2 + (\alpha_{11}\alpha_{32}(qc-p)+\alpha_{31}\alpha_{12}(1-qcp))y_2y_1 \\
	&+ (\alpha_{11}\alpha_{33}(q-p)+\alpha_{31}\alpha_{13}(1-qp)+(\alpha_{12}\alpha_{34}-p\alpha_{32}\alpha_{14})q(1-cd))y_3y_1 \\
	&+ (\alpha_{11}\alpha_{34}(c\inv - p) + \alpha_{31}\alpha_{14}(1-c\inv p))y_4y_1 \\
	&+ (\alpha_{12}\alpha_{33}(qd\inv- p)+\alpha_{32}\alpha_{13}(1-qd\inv p))y_3y_2 \\
	&+ (1-p)(\alpha_{12}\alpha_{34}+\alpha_{32}\alpha_{14})y_4y_2 \\
	&+ (\alpha_{13}\alpha_{34}(d\inv- p)+\alpha_{33}\alpha_{14}(1-d\inv p))y_4y_3.
\end{align*}
We claim $\alpha_{12},\alpha_{14},\alpha_{32},\alpha_{34}=0$. Suppose to the contrary that $\alpha_{12} \neq 0$. Since the coefficient of $y_d^2$ in $T_2(1,3)$ is zero for each $d\in\{1,\hdots,4\}$, and because $p\neq 1$ and $\alpha_{12}\neq 0$, then $\alpha_{32}=0$. Now the coefficient of $y_4y_2$ is $(1-p)\alpha_{12}\alpha_{34}$, and so $\alpha_{34}=0$.

Repeating this argument with $T_2(1,2)$ and $T_2(1,4)$ we have $\alpha_{22}=\alpha_{24}=\alpha_{42}=\alpha_{44}=0$. But then $\dim(\Span\{\Phi_1(x_2),\Phi_1(x_3),\Phi_1(x_4)\})=2$, contradicting Lemma \ref{deg1}. Thus, $\alpha_{12}=0$. A similar argument shows $\alpha_{14},\alpha_{32},\alpha_{34}=0$.

Now $T_2(1,3)=(\alpha_{11}\alpha_{33}(q-p)+\alpha_{31}\alpha_{13}(1-qp))y_3y_1$. If $\alpha_{11}=\alpha_{13}=0$ or $\alpha_{33}=\alpha_{31}=0$, then we contradict Lemma \ref{deg1}. If $\alpha_{11}=\alpha_{31}=0$, then $\alpha_{13}\inv\Phi_1(x_1) = y_3 = \alpha_{33}\inv\Phi_1(x_3)$. Thus, $0 = \Phi_1(\alpha_{13}\inv x_1 - \alpha_{33} x_3)$ and, by Lemma \ref{deg1}, $\alpha_{13}\inv x_1 - \alpha_{33} x_3=0$, contradicting the linear independence of $x_1$ and $x_3$. We arrive at a similar contradiction if we assume $\alpha_{33}=\alpha_{13}=0$. Thus, either $\alpha_{11}\alpha_{33}\neq 0$, in which case $q=p$, or else $\alpha_{31}\alpha_{13} \neq 0$, in which case $p=q\inv$. By our assumption that $p^2 \neq 1$, these both cannot hold.
	
Case 1 ($p=q$) In this case, $\Phi_1(x_1)=\alpha_{11}y_1$ and $\Phi_1(x_3)=\alpha_{33}y_3$ with $\alpha_{11},\alpha_{33} \neq 0$. For $i \neq j$, the coefficient of $y_d^2$ in $T_2(i,j)$ is zero. Thus, $\alpha_{41}=\alpha_{43}=0$ and so
\begin{align*}
	T_2(4,1) &= (\alpha_{42}y_2+\alpha_{44}y_4)\alpha_{11}y_1 - a \alpha_{11}y_1 (\alpha_{42}y_2+\alpha_{44}y_4) \\
		&= \left[ \alpha_{42}(1-aqc)y_2 + \alpha_{44}(1-ac\inv)y_4 \right]\alpha_{11}y_1.
\end{align*}
If $\alpha_{42}$ and $\alpha_{44}$ are both nonzero, then $1=aqc$ and $1=ac\inv$ implying $qc^2=1$, contradicting our hypothesis. Hence, either $a=c$ or $a=(qc)\inv$, and, depending on the choice, $T_2(3,4)$ implies $b=d$ or $b=q d\inv$, respectively.

Case 2 ($p=q\inv$) In this case, $\Phi_1(x_1)=\alpha_{13}y_3$ and $\Phi_1(x_3)=\alpha_{31}y_1$ with $\alpha_{13},\alpha_{31} \neq 0$. Then, as in Case 1, $\alpha_{41}=\alpha_{43}=0$ and so
\begin{align*}
	T_2(4,1) &= (\alpha_{42}y_2+\alpha_{44}y_4)\alpha_{13}y_3 - a\alpha_{13}y_3 (\alpha_{42}y_2+\alpha_{44}y_4) \\
		&= \left[ \alpha_{42}(1-aq\inv d)y_2 + \alpha_{44}(1-ad\inv)y_4 \right]\alpha_{13}y_3.
\end{align*}
If $\alpha_{42}$ and $\alpha_{44}$ are both nonzero, then $q=ad$ and $a=d$ implying $q=d^2$, contradicting our hypothesis. Hence, either $a=d$ or $a=qd\inv$, and, depending on the choice, the commutation relation for $y_4$ and $y_3$ implies $b=c$ or $b=(qc)\inv$, respectively.
\end{proof}

The problem with applying this approach to the general case $(p_2,q_2 \neq 1)$ is that $\alpha_{12}\neq 0$ no longer implies $\alpha_{34}=0$. Further restrictions on the defining parameters would allow this proof to carry through. Otherwise, it seems clear that another approach will be necessary.

\section{Jordan matrix algebra}

In this section we give a brief proof that $\jmn$ is not isomorphic to the matrix algebras discussed above. We cannot apply Lemma \ref{dzero}, but we can achieve a similar result that will be sufficient for these purposes.

\begin{prop}\label{jmn}The algebra $\jmn$ is not isomorphic to $R(c,d,q_1,q_2)$.\end{prop}

\begin{proof}Suppose to the contrary that $\Phi:\jmn \rightarrow R(c,d,q_1,q_2)$ is an isomorphism. Let $\{x_i\}$ be the generating basis for $\jmn$ and $\{y_i\}$ that for $R(c,d,q_1,q_2)$. Let $T=\Phi(x_1)\Phi(x_3)-\Phi(x_3)\Phi(x_1)+\Phi(x_3)^2$. Since $T_0=0$, then $\Phi_0(x_3)^2=0$. Thus, $\Phi_0(x_3)=0$ and so $T_1=0$. Now,
\begin{align*}
	T_2 &= (\Phi_1(x_1)\Phi_1(x_3)+\Phi_0(x_1)\Phi_2(x_3))-(\Phi_1(x_3)\Phi_1(x_1)+\Phi_2(x_3)\Phi_0(x_1))
		+\Phi_1(x_3)^2 \\
	&= \Phi_1(x_1)\Phi_1(x_3)-\Phi_1(x_3)\Phi_1(x_1)+\Phi_1(x_3)^2.
\end{align*}
Write $\Phi_1(x_1)=\sum \alpha_i y_i$ and $\Phi_1(x_3)=\sum \beta_i y_i$. Then
\begin{align*}
	T_2 &= \sum_{k=1}^4 \beta_k^2 y_k^2 + \sum_{1 \leq i \neq j \leq 4} (\alpha_i\beta_j - \alpha_j\beta_i + \beta_i\beta_j)y_iy_j.
\end{align*}
Because the commutation relations in $R(c,d,q_1,q_2)$ for $y_iy_j$ do not involve $y_k^2$, $k \in \{1,\hdots,4\}$, then $T_2=0$ implies $\beta_k=0$ for all $k$, contradicting Lemma \ref{deg1}.
\end{proof}

\section{Quantum Weyl algebras}\label{qpwa}

In this section we assume $\ch\K=0$. In this case, the quantum Weyl algebra, $\wa$, is simple if and only if $q=1$. Moreover, $\Aut(\wa) \iso \K$ unless $q=\pm 1$ \cite{alev1}. Thus, there is no loss in assuming henceforth that $p,q \neq \pm 1$.

Let $\{X,Y\}$ (resp. $\{x,y\}$) be a generating basis for $\wap$ (resp. $\wa$) and define the normal elements $Z=XY-YX \in \wap$ and $z=xy-yx \in \wa$.

\begin{prop}If $p=q^{\pm 1}$, then $\wap \iso \wa$.\end{prop}

\begin{proof}If $p=q$, then there is nothing to prove. If $p=q\inv$, then define a rule by $\theta(X)=qy$ and $\theta(Y)=-x$. Then,
    \[ \theta(X)\theta(Y) - q\inv \theta(Y)\theta(X) - 1 = -qyx + xy - 1=0.\]
Hence, $\theta$ extends to a homomorphism $\wap \rightarrow \wa$. Moreover, the map is bijective and therefore an isomorphism.\end{proof}

Recall that $\wa$ is PI if and only if $q$ is a primitive root of unity of order $\ell$, in which case $\cnt(\wa)=\K[x^\ell,y^\ell]$, and otherwise $\cnt(\wa)=\K$ (\cite{awami}, Lemma 2.2).
Hence, we consider the nonroot and root of unity cases separately (Propositions \ref{wiso1} and \ref{wiso2}, respectively). The nonroot of unity case actually follows from \cite{adcorps}, Proposition 3.11. However the proof given here is more direct and is re-used in Proposition \ref{wiso2}.

Let $\cnt(A)$ denote the center of the algebra $A=\wa$ or $A=\wap$. Throughout the remainder of this section, assume $\theta:\wap \rightarrow \wa$ is an isomorphism. By \textit{degree} we mean total degree in $X$ and $Y$ in $\wap$ (resp. $x$ and $y$ in $\wa$). The next lemma can be thought of as an ungraded version of Lemma \ref{deg1}.

\begin{lem}\label{wdeg}$\deg(\theta(X)),\deg(\theta(Y)) \geq 1$.\end{lem}

\begin{proof}Without loss of generality, suppose $\deg(\theta(X))=0$. Then $\theta(X) \in \cnt(\wa)$, implying $X \in \cnt(\wap)$. This cannot hold by the above discussion.\end{proof}

\begin{prop}\label{wiso1}Let $p,q \in \K^\times$ with $p,q$ non-roots of unity. If $\wap \iso \wa$, then $p=q^{\pm 1}$.\end{prop}

\begin{proof}By \cite{good}, Theorem 8.4 (a), the intersection of all nonzero prime ideals in $\wap$ (resp. $\wa$) is $Z\wap$ (resp. $z\wa$). Hence, $\theta(Z\wap)=\theta(Z)\theta(\wap)=\theta(Z)\wa$. Since $\theta(Z)\in z\wa$, then $\theta(Z)=\lambda z$ for some $\lambda \in \wa$. We claim $\lambda \in \K^\times$. The ideal $z\wa$ is generated by $z$, so there exists $g \in \wa$ such that $g\cdot \lambda z=z$. Hence, $\lambda$ is a unit in $\wa$ and therefore $\lambda \in \K^\times$. This gives $\theta(Z) = \lambda z=\lambda(xy-yx)=\lambda(q-1)yx + \lambda$, and so, \[\theta(X)\theta(Y) = \theta(Y)\theta(X) + \lambda(q-1)yx + \lambda.\] Since $\theta$ is an isomorphism,
\begin{align*}
    0 &= \theta(XY-pYX-1) = \theta(X)\theta(Y)-p\theta(Y)\theta(X)-1 \\
        &= \left(\theta(Y)\theta(X) + \lambda(q-1)yx + \lambda\right) - p\theta(Y)\theta(X)-1 \\
        &= (1-p)\theta(Y)\theta(X) + \lambda(q-1)yx + (\lambda-1), 
\end{align*}
and so,
\begin{align}\label{wa1}
    \theta(Y)\theta(X) &= (p-1)\inv \left(\lambda(q-1)yx +(\lambda-1)\right).
\end{align}
We claim $\deg\theta(X)=\theta(Y)=1$ in $\wa$. Write $\theta(X)=a=a_0 + \cdots a_n$, $a_n \neq 0$, and $\theta(Y)=b=b_0+\cdots b_m$, $b_m \neq 0$, wherein $a_d$ is the sum of the monomomials of total degree $d$ written according to the filtration $\{y^i x^j \mid i,j \in \mathbb{N}\}$ (and similarly for $b_d$). Because $\wa$ is a domain, the highest degree component of $\theta(Y)\theta(X)$ is $b_ma_n \neq 0$. If $n$ or $m$ is greater than 1, then the left hand side of \eqref{wa1} will have degree greater than 2, a contradiction. This proves the claim. Thus, we can write $\theta(X)=\alpha x + \beta y + \gamma$ and $\theta(Y)=\alpha' x + \beta' y + \gamma'$, $\alpha,\alpha',\beta,\beta',\gamma,\gamma' \in \K$. Substituting this into \eqref{wa1} gives
\begin{align}\label{wrel}
\notag	\alpha'\alpha x^2 + \alpha'\beta xy &+ \alpha'\gamma x + \beta'\alpha yx + \beta'\beta y^2 + \beta'\gamma y + \gamma'\alpha x + \gamma'\beta y + \gamma'\gamma \\
	&= \lambda\frac{q-1}{p-1}yx + \frac{\lambda-1}{p-1}.
\end{align}
Thus, $\alpha'\alpha=\beta'\beta=0$. If $\alpha=\beta=0$, then $\theta(X)$ is a constant and similarly for $\theta(Y)$ if $\alpha'=\beta'=0$. This contradicts Lemma \ref{wdeg}.

If $\alpha'=\beta=0$, then \eqref{wrel} reduces to
    \[ \beta'\alpha yx + \beta'\gamma y + \gamma'\alpha x + \gamma'\gamma = (p-1)\inv (\lambda(q-1)yx + (\lambda-1)).\]
Thus, $\beta'\alpha \neq 0$ but $\beta'\gamma=\gamma'\alpha=0$ so $\gamma=\gamma'=0$. This holds only if $\lambda=1$ so
\begin{align*}
    0 &= \theta(XY-pYX-1) = \beta'\alpha(xy-pyx)-1 \\
        &= \beta'\alpha(qyx+1-pyx)-1
        = \beta'\alpha(q-p)yx+(\beta'\alpha-1).
\end{align*}
Therefore, $p=q$.

Otherwise, $\alpha=\beta'=0$ and \eqref{wrel} reduces to
\begin{align*}
    \alpha'\beta xy + \alpha'\gamma x + \gamma'\beta y + \gamma'\gamma &= (p-1)\inv (\lambda(q-1)yx + (\lambda-1))\\
  \alpha'\beta (qyx+1) + \alpha'\gamma x + \gamma'\beta y + \gamma'\gamma &= (p-1)\inv (\lambda(q-1)yx + (\lambda-1)) \\
  q\alpha'\beta yx + \alpha'\gamma x + \gamma'\beta y + (\alpha'\beta + \gamma'\gamma) &= (p-1)\inv (\lambda(q-1)yx + (\lambda-1)).
\end{align*}
As above, $\gamma=\gamma'=0$ so
\begin{align*}
    0 &= \theta(XY-pYX-1) = \alpha'\beta(yx-pxy)-1 \\
        &= \alpha'\beta(yx-p(qyx+1))-1
        = \alpha'\beta(1-pq)yx-(p\alpha'\beta+1).
\end{align*}
Therefore, $p=q\inv$.
\end{proof}

\begin{prop}\label{wiso2}Let $p,q \in \K^\times$ with $p,q \neq \pm 1$ primitive roots of unity. If $\wap \iso \wa$, then $p=q^{\pm 1}$.\end{prop} 

\begin{proof}As in Proposition \ref{wiso1}, write $\theta(X)=a=a_0 + \cdots + a_n$ and $\theta(Y)=b=b_0+\cdots+b_m$, $a_n,b_m \neq 0$. By Lemma \ref{wdeg}, $m+n>0$. We decompose $a_n$ and $b_m$ further as
 \[ a_n = \sum_{i=0}^n a_{n,i} y^{n-i}x^i, ~~ b_m = \sum_{j=0}^m b_{m,j} y^{m-j}x^j, ~~ a_{n,i},b_{m,j} \in \K \text{ for all } i,j.\]
Choose $r,s$ minimal such that $a_{n,r},b_{m,s}\neq 0$. As $0=\theta(XY-pYX-1)=ab-pba-1$, the highest y-degree term in $a_nb_m-pb_ma_n$ is \[a_{n,r}b_{m,s}\left[q^{r(m-s)}-pq^{s(n-r)}\right] y^{n+m-r-s}x^{r+s}=0.\] Hence, $q^{r(m-s)}-pq^{s(n-r)} = q^{r(m-s)}(1-pq^{ns-mr})=0$. This implies that
\begin{align}\label{params}
	p=q^{mr-ns}.
\end{align} 
Likewise, $q=p^t$ for some $t \in \mathbb{N}$. Thus, $p$ and $q$ are roots of unity of the same order $\ell$. Hence, $\cnt(\wap)=\K[X^\ell,Y^\ell]$ and $\cnt(\wa)=\K[x^\ell,y^\ell]$. Then $\theta(X^\ell) = a^\ell = a_{n\ell}' + a_{n\ell-1}' + \cdots a_0'$ where $a_d'$ is the term of $a^\ell$ of total degree $d$. Thus,
\begin{align}\label{wa2}
     a_{n\ell}' = \alpha_{n,r}^\ell q^v y^{(n-r)\ell} x^{r\ell} + \sum_{j=0}^{r\ell-1} \alpha_{n\ell,j}' y^{n\ell-j}x^j,
\end{align}
with $v \in \mathbb{Z}$ and $\alpha_{n\ell,j}' \in \K$. Similarly, $\theta(Y^\ell) = b^\ell = b_{m\ell}' + b_{m\ell-1}' + \cdots b_0'$ where
\begin{align}\label{wa3}
     b_{m\ell}' = \beta_{m,s}^\ell q^w y^{(m-s)\ell} x^{s\ell} + \sum_{j=0}^{s\ell-1} \beta_{m\ell,j}' y^{m\ell-j}x^j.
\end{align}
The restriction of $\theta$ to the centers of the respective algebras determines an automorphism of the polynomial ring in two variables. The centrality of $X^\ell$ and $Y^\ell$ implies $\theta(X^\ell)$ and $\theta(Y^\ell)$ are central. Thus, $a_e'=b_e'=0$ if $e \not\equiv 0$ modulo $\ell$ and $\alpha_{n\ell,j}'=\beta_{m\ell,j}'=0$ if $j \not\equiv 0$ modulo $\ell$.
Lemma 2 of \cite{makar} shows that there are three possibilities for an automorphism of the polynomial ring in two variables (see also \cite{alev1}).

Case 1: There exists $t \in \mathbb{Z}_{> 0}$ and $\lambda \in \K$ such that $a_{n\ell}'=\lambda(b_{m\ell}')^t$. Substituting into \eqref{wa2} and \eqref{wa3} shows that $r=st$ and $n=mt$, so $ns=mr$. Then \eqref{params} implies $p=1$, a contradiction.

Case 2: There exists $t \in \mathbb{Z}_{> 0}$ and $\lambda \in \K$ such that $b_{m\ell}'=\lambda(a_{n\ell}')^t$. This gives the same contradiction as above.

Case 3: $\theta(X^\ell) = \zeta x^\ell + \xi y^\ell + \omega$ and $\theta(Y^\ell) = \zeta' x^\ell + \xi' y^\ell + \omega'$ with $\zeta,\xi,\omega,\zeta',\xi',\omega' \in \K$. Hence, $\deg\theta(X)=\deg\theta(Y)=1$ and we refer to the proof of Proposition \ref{wiso1}.
\end{proof}

\section{Appendix: Quantum affine spaces}\label{qas}

For $\bp \in \A_n$, \textit{quantum affine $n$-space} $\qnp$ is defined as the algebra with generating basis $\{x_i\}$, $1 \leq i \leq n$, subject to the relations $x_ix_j=p_{ij}x_jx_i$ for all $1 \leq i,j \leq n$. The algebra $\qnp$ is affine connected graded and generated in degree $1$. By \cite{browngood}, Lemma II.9.7, $\gk(\qnp)=n$. Hence, if $\qnp \iso \qmq$, then $n=m$. We prove that two quantum affine $n$-spaces, $\qnp$ and $\qnq$, are isomorphic if and only if $\bp$ is a permutation of $\bq$ (Theorem \ref{qas-iso}).

\begin{lem}\label{taudef}Let $\Phi:R \rightarrow S$ be a (not necessarily graded) isomorphism between affine connected graded algebras. Let $\{x_i\}$ (resp. $\{y_i\}$) be a generating basis for $R$ (resp. $S$) and suppose $1 \leq i \leq n$ in both cases. The isomorphism $\Phi$ determines a permutation $\tau \in \sym_n$.\end{lem}

\begin{proof}Write $\Phi_1(x_i)=\sum \gamma_{ij} y_j$ and let $M = (\gamma_{ij})$. By Lemma \ref{deg1}, $\Phi_1:R_1 \rightarrow S_1$ is a vector space isomorphism, and so $\det(M) \neq 0$. The case of $n=1$ is trivial. We proceed by induction. Let $M_j$ be the minor of $M$ corresponding to the entry $\gamma_{1j}$. Then,
	\[ \det(M) = \sum_{j=1}^{n} (-1)^{j+1} \gamma_{1j} \det(M_j).\]
Since $\det(M) \neq 0$, there exists $\tau(i) \in \{1,\hdots,n\}$ such that $\gamma_{i\tau(i)} \det(M_{\tau(i)}) \neq 0$. We pass to $M_{\tau(i)}$ and, because $\dim(M_{\tau(i)}) = (n-1)^2$, the result follows by the inductive hypothesis.\end{proof}

For the remainder, let $\{x_i\}$ (resp. $\{y_i\}$) be a generating basis for $\qnp$ (resp. $\qnq$) and suppose $\Phi:\qnp \rightarrow \qnq$ is an isomorphism. By Lemma \ref{taudef}, $\Phi$ gives a permutation $\tau \in \sym_n$. It suffices to show that $\bp = \tau.\bq$. 

\begin{lem}\label{qparam1}If $r,s \in \{1,\hdots,n\}$ such that $p_{rs} \neq 1$, then $p_{rs}=q_{\tau(r)\tau(s)}$.\end{lem}

\begin{proof}By Lemma \ref{dzero}, $\Phi_0(x_r)=\Phi_0(x_s)=0$. Write $\Phi_1(x_r)=\sum \alpha_i y_i$ and $\Phi_1(x_s)=\sum \beta_i y_i$. Let $T=T(x_r,x_s)$. Then, 
\begin{align*}
0 = T_2
	= (1-p_{rs})\left(\sum_{d=1}^n \alpha_d \beta_d y_d^2\right) + \sum_{1 \leq i \neq j \leq n} \left(\alpha_i \beta_j -p_{rs} \alpha_j\beta_i \right) y_iy_j.
\end{align*}
Since $p_{rs} \neq 1$, then $\alpha_d=0$ or $\beta_d = 0$ for each $d$. Thus,
\begin{align}
  	T_2
  		&= \sum_{1 \leq i < j \leq n} \left[ (\alpha_i \beta_j -p_{rs} \alpha_j\beta_i) + q_{ji}(\alpha_j\beta_i - p_{rs}\alpha_i\beta_j)\right] y_iy_j \notag \\
  		&= \sum_{1 \leq i < j \leq n} \left[ (\alpha_j\beta_i(q_{ji}-p_{rs}) + \alpha_i\beta_j (1 - q_{ji}p_{rs})\right]y_iy_j. \label{not1}
\end{align}
By Lemma \ref{taudef}, $\alpha_{\tau(r)}$, $\beta_{\tau(s)} \neq 0$. Thus, $\alpha_{\tau(s)}=0$ and $\beta_{\tau(r)}=0$. If $\tau(r) > \tau(s)$, then by \eqref{not1} the coefficient of $y_{\tau(s)}y_{\tau(r)}$ is $\alpha_{\tau(r)}\beta_{\tau(s)}(q_{\tau(r)\tau(s)} - p_{rs})$. Therefore, $p_{rs} = q_{\tau(r)\tau(s)}$. One the other hand, if $\tau(r) < \tau(s)$, then the coefficient of $y_{\tau(r)}y_{\tau(s)}$ is $\alpha_{\tau(r)}\beta_{\tau(s)}(1-q_{\tau(s)\tau(r)}p_{rs})$. Therefore, $p_{rs} = q_{\tau(s)\tau(r)}\inv = q_{\tau(r)\tau(s)}$. Because $p_{rs} \neq 1$, then $r \neq s$ and so, because $\tau$ is a permutation, $\tau(r) \neq \tau(s)$ and so the result follows.
\end{proof}

For $\bp \in \A_n$, let $\bp^{\#}=\{p_{ij} \in \bp \mid p_{ij} \neq 1\}$.

\begin{lem}\label{qparam2}If $r,s \in \{1,\hdots,n\}$ such that $p_{rs} = 1$, then $p_{rs}=q_{\tau(r)\tau(s)}$.\end{lem}

\begin{proof}By Lemma \ref{qparam1}, $\bp^{\#} \leq \bq^{\#}$. Because $\Phi$ is an isomorphism, then we can apply Lemma \ref{qparam1} to $\Phi\inv$ to get that $\bq^{\#} \leq \bp^{\#}$. Thus, $\bp^{\#} = \bq^{\#}$.
\end{proof}

\begin{thm}\label{qas-iso}$\qnp \iso \qnq$ if and only $\bp$ is a permutation of $\bq$.\end{thm}

\begin{proof}Suppose there exists $\sigma \in \sym_n$ such that $\bp=\sigma.\bq$. We wish to define a homomorphism $\qnp \rightarrow \qnq$ via the rule $\Psi(x_i) = y_{\sigma(i)}$. For all $i,j$, $1 \leq i,j \leq n$, this rule gives
	\[ \Psi(x_i)\Psi(x_j) - p_{ij}\Psi(x_j)\Psi(x_i) = y_{\sigma(i)}y_{\sigma(j)} - q_{\sigma(i)\sigma(j)}y_{\sigma(j)}y_{\sigma(i)}=0. \]
Hence, $\Psi$ extends to a bijective homomorphism. Thus, $\qnp \iso \qnq$.

Conversely, suppose $\qnp \iso \qnq$. Lemma \ref{taudef} gives a permutation $\tau \in \sym_n$. By Lemmas \ref{qparam1} and \ref{qparam2}, $\bp = \tau.\bq$.\end{proof}

\begin{cor}\label{qp-iso}$\qpp \iso \qp$ are isomorphic if and only if $p = q^{\pm 1}$.\end{cor}

\section*{Acknowledgements}

The author would like to thank the organizers of the 31st Ohio State-Denison Conference where parts of this work were first presented. He thanks his advisor, Allen Bell, for his patient guidance throughout this project, and the referee for many corrections and suggestions.

\bibliographystyle{amsplain}
\bibliography{biblio}{}

\end{document}